\documentclass[a4paper,11pt,fleqn,leqno]{article}

\usepackage{etex}
\usepackage[utf8]{inputenc}
\usepackage[english]{babel}
\usepackage{lmodern}
\usepackage[T1]{fontenc}
\usepackage{makeidx}
\usepackage{ragged2e}
\usepackage{tabularx}
\usepackage{enumitem}
\usepackage[nodayofweek]{datetime}
\usepackage{wrapfig}
\usepackage{cclicenses}
\usepackage[bottom]{footmisc}

\usepackage{amssymb,amsfonts,euscript,mathrsfs,nicefrac,cancel,multirow,centernot,natbib}
\usepackage[fleqn,leqno]{mathtools}
\usepackage[thmmarks,amsmath,hyperref]{ntheorem}
\usepackage[only,mapsfrom,sslash,Yright]{stmaryrd}
\SetSymbolFont{stmry}{bold}{U}{stmry}{m}{n} 

\usepackage[all]{xy}
\input xy
\SelectTips{cm}{10}
\newdir{ >}{{}*!/-5pt/@{>}}
\newdir{c}{{}*!/-5pt/@{(}}
\newdir^{c}{{}*!/-5pt/@^{(}}
\newdir_{c}{{}*!/-5pt/@_{(}}
\entrymodifiers={+!!<0pt,\the\fontdimen22\textfont2>} 

\newif\ifpdflatex
\ifx\pdfoutput\undefined
  \pdflatexfalse
\else
  \ifnum \pdfoutput=0
    \pdflatexfalse
  \else
    \ifnum\pdfoutput=1
    \pdflatextrue
    \fi
  \fi
\fi

\usepackage{graphicx}
\ifpdflatex
	\usepackage[pdftex,
							colorlinks, 
							bookmarksnumbered, 
							bookmarksopen, 
							pdfstartview=FitH, 
							linkcolor=blue, 
							citecolor=blue, 
							urlcolor=magenta, 
							filecolor=blue %
							]{hyperref}
\else
  \usepackage[dvips,
  							colorlinks, 
  							bookmarksnumbered, 
  							bookmarksopen, 
  							pdfstartview=FitH, 
  							linkcolor=blue, 
  							citecolor=blue, 
  							urlcolor=magenta, 
  							filecolor=blue]{hyperref}
\fi
\ifcase\pdfoutput
	\usepackage{breakurl}
\else\fi

\usepackage{etoolbox}
\apptocmd{\sloppy}{\hbadness 10000\relax}{}{}

\bibpunct{[}{]}{;}{n}{,}{,}	

\makeatletter
\def\blfootnote{\xdef\@thefnmark{}\@footnotetext}
\makeatother

\usepackage[outer=2.5cm,inner=2.5cm,nofoot,bottom=3.5cm,top=2.8cm]{geometry}

\setlist[1]{leftmargin=\mathindent}

\usepackage{fancyhdr}
\usepackage{titlesec}
\titleformat{\section}%
	{\centering\large\scshape}%
	{\thesection.}%
	{1em}%
	{}%
\titlespacing*{\section}{0pt}{3.5ex plus 1ex minus .2ex}{2.3ex plus .2ex}

\titleformat{\subsection}%
	{}%
  {\thesubsection.}%
	{1em}%
  {\bf}%
\titlespacing*{\subsection}{0pt}{3.5ex plus 1ex minus .2ex}{2.3ex plus .2ex}

\let\originallabel\label
\def\label{\phantomsection\originallabel}

\def\letters{a,b,c,d,e,f,g,h,i,j,k,l,m,n,o,p,q,r,s,t,u,v,w,x,y,z}
\def\Letters{A,B,C,D,E,F,G,H,I,J,K,L,M,N,O,P,Q,R,S,T,U,V,W,X,Y,Z}
\makeatletter
\@for \@l:=\Letters \do{%
  \expandafter\edef\csname\@l bb\endcsname{\noexpand\ensuremath{\noexpand\mathbb{\@l}}}%
  \expandafter\edef\csname\@l bf\endcsname{{\noexpand\bf \@l}}%
  \expandafter\edef\csname\@l cal\endcsname{\noexpand\ensuremath{\noexpand\mathcal{\@l}}}%
  \expandafter\edef\csname\@l eu\endcsname{\noexpand\ensuremath{\noexpand\EuScript{\@l}}}%
  \expandafter\edef\csname\@l frak\endcsname{\noexpand\ensuremath{\noexpand\mathfrak{\@l}}}%
  \expandafter\edef\csname\@l rm\endcsname{{\noexpand\rm \@l}}%
  \expandafter\edef\csname\@l scr\endcsname{\noexpand\ensuremath{\noexpand\mathscr{\@l}}}%
}
\@for \@l:=\letters \do{%
  \expandafter\edef\csname\@l bf\endcsname{{\noexpand\bf \@l}}%
  \expandafter\edef\csname\@l frak\endcsname{\noexpand\ensuremath{\noexpand\mathfrak{\@l}}}%
  \expandafter\edef\csname\@l rm\endcsname{{\noexpand\rm \@l}}%
}
\makeatother

\newcommand*{\N}{\ensuremath{\mathbb{N}}}

\let\tempepsilon\epsilon
\let\epsilon\varepsilon
\let\varepsilon\tempepsilon
\let\tempphi\phi
\let\phi\varphi
\let\varphi\tempphi
\renewcommand*\geq{\geqslant}
\renewcommand*\leq{\leqslant}
\DeclareFontFamily{U}{MnSymbolC}{}
\DeclareFontShape{U}{MnSymbolC}{m}{n}{
    <-6>  MnSymbolC5
   <6-7>  MnSymbolC6
   <7-8>  MnSymbolC7
   <8-9>  MnSymbolC8
   <9-10> MnSymbolC9
  <10-12> MnSymbolC10
  <12->   MnSymbolC12}{}
\DeclareFontShape{U}{MnSymbolC}{b}{n}{
    <-6>  MnSymbolC-Bold5
   <6-7>  MnSymbolC-Bold6
   <7-8>  MnSymbolC-Bold7
   <8-9>  MnSymbolC-Bold8
   <9-10> MnSymbolC-Bold9
  <10-12> MnSymbolC-Bold10
  <12->   MnSymbolC-Bold12}{}
\DeclareSymbolFont{MnSyC}{U}{MnSymbolC}{m}{n}
\SetSymbolFont{MnSyC}{bold}{U}{MnSymbolC}{b}{n}
\DeclareMathSymbol{\lefthalfcup}{\mathord}{MnSyC}{183}
\DeclareMathSymbol{\righthalfcup}{\mathord}{MnSyC}{184}
\DeclareMathSymbol{\lefthalfcap}{\mathord}{MnSyC}{185}
\DeclareMathSymbol{\righthalfcap}{\mathord}{MnSyC}{186}
\DeclareMathSymbol{\medsquare}{\mathord}{MnSyC}{106}
\DeclareFontFamily{U}{mathb}{\hyphenchar\font45}
\DeclareFontShape{U}{mathb}{m}{n}{
      <5> <6> <7> <8> <9> <10> gen * mathb
      <10.95> mathb10 <12> <14.4> <17.28> <20.74> <24.88> mathb12
      }{}
\DeclareSymbolFont{mathb}{U}{mathb}{m}{n}
\DeclareFontSubstitution{U}{mathb}{m}{n}
\DeclareMathSymbol{\Asterisk}{2}{mathb}{"06}
\newcommand*{\bigast}{\mathop{\Asterisk}}
\newcommand*{\defas}{\mathrel{\mathrel{\mathop:}=}}

\newcommand*{\rar}{\ensuremath{\Rightarrow}}

\newcommand*{\set}[2]{\left\{#1\;\middle\vert\;#2\right\}} 
\renewcommand*{\emptyset}{\varnothing}

\newcommand*{\Const}{{\rm Const}}

\newcommand*{\holim}{\mathop{\rm holim}}

\newcommand*{\hocolim}{\mathop{\rm hocolim}}

\newcommand*{\Fib}{\mathop{{\rm h}\mathfrak{F}{\rm ib}}\nolimits}

\newcommand*{\hFib}{\mathop{\rm hFib}\nolimits}
\newcommand*{\hfib}{\mathop{\rm hFib}\nolimits}
\newcommand*{\Cof}{\mathop{\rm Cof}\nolimits}

\newcommand*{\Kan}{\mathop{\rm Kan}\nolimits}

\newcommand*{\adjRelayII}[3][2.2em]{\ensuremath{\SelectTips{cm}{10}\xymatrix@C=#1@1{*++{#2} \ar@<1ex>[r]^-{\ArgI}^-{}="1" & *++{#3} \ar@<1ex>[l]^-{\ArgII}^-{}="2" \ar@{}"1";"2"|(.3){\hbox{}}="3"|(.7){\hbox{}}="4" \ar@{-|}"3";"4"}}}

\newcommand*{\radjRelayII}[3][2.2em]{\ensuremath{\SelectTips{cm}{10}\xymatrix@C=#1@1{*++{#2} \ar@<-1ex>[r]_-{\ArgII}^-{}="1" & *++{#3} \ar@<-1ex>[l]_-{\ArgI}^-{}="2" \ar@{}"1";"2"|(.3){\hbox{}}="3"|(.7){\hbox{}}="4" \ar@{-|}"4";"3"}}}
\newcommand*{\twocell}[6][]{\xymatrix#1{#2\ar@/^1pc/[r]^-{#4}="1" \ar@/_1pc/[r]_-{#5}="2" & #2 \ar@{}"1";"2"|(.3){}="3" \ar@{}"1";"2"|(.7){}="4" \ar@{=>}_-{#6} "3" ;"4"}}

\newcommand*{\quintetRelay}[9]{\ensuremath{\SelectTips{cm}{10}\xymatrix{#1 \ar[r]^{#5} \ar[d]_{#6} & #2 \ar[d]^{#7} \ar@{}[dl]|(.4){\hbox{}}="1"|(.6){\hbox{}}="2" \\ #3 \ar[r]_{#8} & #4 & *!<2em,.5ex>{\ArgI} \ar@{=>}"1";"2"_{#9}}}}

\renewcommand\coprod\amalg

\newcommand*{\sSets}{{\bf sSets}}

\newcommand*{\Cat}{{\bf Cat}}

\newcommand*{\App}{\mathop{\rm App}\nolimits}

\newcommand*{\del}{\partial}

\newcommand*{\eop}{\leavevmode\unskip\penalty9999 \hbox{}\nobreak\hfill\hbox{\ensuremath{_\Box}}}
\makeatletter
\newcommand*{\dashbar}{\mathrel{\dabar@ \dabar@ \dabar@}}
\makeatother

\makeatletter
\newcommand*{\enumnobreak}{\par\vskip.65\baselineskip\nobreak\@afterheading}

\makeatother
\newcommand*{\hyph}{\ifmmode\hbox{-}\else\nobreakdash-\hspace{0pt}\fi}

\renewcommand*{\cases}[1]{\ensuremath{\left\{\begin{array}{@{}ll} #1 \end{array} \right.}}

\let\dummycolon\colon
\def\colon{\leavevmode\ifmmode\dummycolon\else: \fi}

\let\dummysslash\sslash
\renewcommand*{\sslash}{\!\dummysslash\!}

\newcommand*{\from}{\leftarrow}
\newcommand*{\xto}[2][]{\xrightarrow[#1]{#2}}

\makeatletter
\def\empty{} 
\newbox\tempboxa
\newbox\tempboxb
\newdimen\monokern
\newcommand{\mono}[2][]{
  \def\tempa{#1}%
  \def\tempb{#2}%
  \ifx\tempa\empty%
    \ifx\tempb\empty%
      \rightarrowtail%
    \else%
      \xrightarrowtail[#1]{#2}%
    \fi%
  \else%
    \xrightarrowtail[#1]{#2}%
  \fi%
}

\newcommand\xrightarrowtail[2][]{%
  \setbox\tempboxa\hbox{$#1$}%
  \setbox\tempboxb\hbox{$#2$}%
  \ifdim\wd\tempboxa < \wd\tempboxb%
    \advance\monokern by -\wd\tempboxb%
  \else%
    \advance\monokern by -\wd\tempboxa%
  \fi%
  \ifdim\monokern < 0pt \monokern=0pt \fi%
  \ext@arrow 0359{\rightarrowtailfill@}{#1}{#2}
}
\def\rightarrowtailfill@{\arrowfill@{\Yright\kern\monokern}\relbar\rightarrow} 

\newcommand\xtwoheadrightarrow[2][]{\ext@arrow 0359{\twoheadrightarrowfill@}{#1}{#2}}
\def\twoheadrightarrowfill@{\arrowfill@\relbar\relbar\twoheadrightarrow}

\newcommand\xleftarrowtail[2][]{
		\def\test{#1} \def\testt{#2}
		\ifx\testt\empty 
			\ifx\test\empty 
				\ext@arrow 3059{\leftarrowtailfill@}{#1}{\ \ }
			\else 
				\ext@arrow 3059{\leftarrowtailfill@}{#1\ }{}
			\fi
		\else 
			\ext@arrow 3059{\leftarrowtailfill@}{#1\ }{#2\ } 
		\fi
}
	\def\leftarrowtailfill@{\arrowfill@\leftarrow\relbar\Yleft}
\newcommand\xtwoheadleftarrow[2][]{\ext@arrow 3059{\twoheadleftarrowfill@}{#1}{#2}}
	\def\twoheadleftarrowfill@{\arrowfill@\twoheadleftarrow\relbar\relbar}

\makeatother

\newcounter{signcounter}[section]
\renewcommand{\thesigncounter}{\arabic{signcounter}}

\newbox\signbox
\newdimen\envindent
\newcommand{\sign}{\setbox\signbox=\hbox{{\mdseries (\thesection.\thesigncounter)}}\copy\signbox}
\newcommand{\signskip}{\unskip\envindent=\mathindent\advance\envindent by -\wd\signbox\hskip\envindent}

\theoremstyle{change}
\theorembodyfont{\normalfont}
\theoremseparator{.}
\newtheorem{axm}{\signskip Axiom}
\newtheorem{conj}{\signskip Conjecture}
\newtheorem{thm}{\signskip Theorem}
\newtheorem{metathm}{\signskip Meta-Theorem}
\newtheorem{mthm}{\signskip Main Theorem}
\newtheorem{lem}{\signskip Lemma}
\newtheorem{cor}{\signskip Corollary}
\newtheorem{prop}{\signskip Proposition}
\newtheorem{metaprop}{\signskip Meta-Proposition}
\newtheorem{obs}{\signskip Observation}
\newtheorem{sch}{\signskip Scholium}
\newtheorem{clm}{\signskip Claim}
\newtheorem{conv}{\signskip Convention}
\newtheorem{defn}{\signskip Definition}
\newtheorem{nota}{\signskip Notation}
\newtheorem{notn}{\signskip Notation}
\newtheorem{nom}{\signskip Nomenclature}
\newtheorem{rem}{\signskip Remark}
\newtheorem{warn}{\signskip Warning}

\newtheorem{ex}{\signskip Example}
\newtheorem{exo}{\signskip Exercise}
\theoremstyle{nonumberplain}
\theoremheaderfont{\it}
\theorembodyfont{\upshape}
\theoremsymbol{\eop}
\newtheorem{proof}{Proof}

\newtagform{empty}{}{}
\usetagform{empty}
\let\dummyaxm\axm\def\axm{\refstepcounter{signcounter}\dummyaxm}
\let\dummyconj\conj\def\conj{\refstepcounter{signcounter}\dummyconj}
\let\dummythm\thm\def\thm{\refstepcounter{signcounter}\dummythm}
\let\dummymetathm\metathm\def\metathm{\refstepcounter{signcounter}\dummymetathm}
\let\dummymthm\mthm\def\mthm{\refstepcounter{signcounter}\dummymthm}
\let\dummylem\lem\def\lem{\refstepcounter{signcounter}\dummylem}
\let\dummycor\cor\def\cor{\refstepcounter{signcounter}\dummycor}
\let\dummyprop\prop\def\prop{\refstepcounter{signcounter}\dummyprop}
\let\dummymetaprop\metaprop\def\metaprop{\refstepcounter{signcounter}\dummymetaprop}
\let\dummyobs\obs\def\obs{\refstepcounter{signcounter}\dummyobs}
\let\dummysch\sch\def\sch{\refstepcounter{signcounter}\dummysch}
\let\dummyclm\clm\def\clm{\refstepcounter{signcounter}\dummyclm}
\let\dummyconv\conv\def\conv{\refstepcounter{signcounter}\dummyconv}
\let\dummydefn\defn\def\defn{\refstepcounter{signcounter}\dummydefn}
\let\dummynota\nota\def\nota{\refstepcounter{signcounter}\dummynota}
\let\dummynotn\notn\def\notn{\refstepcounter{signcounter}\dummynotn}
\let\dummynom\nom\def\nom{\refstepcounter{signcounter}\dummynom}
\let\dummyrem\rem\def\rem{\refstepcounter{signcounter}\dummyrem}
\let\dummywarn\warn\def\warn{\refstepcounter{signcounter}\dummywarn}
\let\dummyex\ex\def\ex{\refstepcounter{signcounter}\dummyex}
\let\dummyexo\exo\def\exo{\refstepcounter{signcounter}\dummyexo}
\let\dummyequation\equation\def\equation{\refstepcounter{signcounter}\dummyequation}

\begin{document}

\thispagestyle{empty}
\begin{center}
  {\Large\textsc{Cubical Acyclic Homotopy Excision}}

  {\textsc{Kay Werndli, EPF Lausanne, \today}}
\end{center}

\begin{abstract}
  Given a strong homotopy pushout cube of spaces~$A$, we measure how far it is from also being a homotopy pullback cube. Explicitly, letting~$P$ be the homotopy colimit of the diagram obtained from~$A$ by forgetting the initial vertex~$A_\emptyset$, we study the homotopy fibre of the double suspension of $A_\emptyset \to P$. This difference is expressible in terms of the homotopy fibres of the original maps in~$A$.
\end{abstract}

\setcounter{section}{-1}
\section{Introduction}
There is a classical duality between homotopy and homology in that homotopy groups are compatible with homotopy pullbacks, while homology groups are compatible with homotopy pushouts. In both cases, we obtain long exact sequences and the classical homotopy excision theorem (which has been proven and generalised time and again by different people \citep{Blakers1952,Brown1987,Ellis1987,GoodwillieII}) answers the question within which range a homotopy pushout induces a long exact sequence of homotopy groups rather than homology groups.

In \citep{Wojciech1997a} and \citep{HEaC}, this classical mere connectivity result for squares is extended to a cellular and acyclic inequality, respectively, which recovers the connectivity statement but potentially allows one to draw broader consequences than just being able to identify an initial range of vanishing homotopy groups. For example, the stronger acyclic inequality is essential in \citep{Chacholski2015}, where it is used to prove a generalisation of ``Bousfield's key lemma'' \citep{Bousfield1994} in order to describe cellular properties of Postnikov sections and spaces in Farjoun's modified Bousfield-Kan tower.

The classical homotopy excision also has analogues for higher-dimensional cubical diagrams (\citep{Brown1987,Ellis1987,GoodwillieII} or the more recent treatment in \citep{Munson2015}) and in this short article, we take a first step towards strengthening these, again obtaining stronger acyclic inequalities, rather than mere connectivity statements. Explicitly, our main result is the following.
\begin{thm}
  Let $A\colon \medsquare^3 \to \sSets$ be a strong homotopy pushout cube of connected spaces, with homotopy fibres $F_k \defas \hfib(A_\emptyset \to A_k)$ and comparison map $q\colon A_\emptyset \to \holim_{\righthalfcup^3} A$. As long as the homotopy fibres $F_1$,~$F_2$ and~$F_3$ are again connected,
  \[ \hFib(\Sigma^2q) > \Sigma(\Omega F_1\ast\Omega F_2\ast\Omega F_3). \]
\end{thm}

Restricting our attention to just connectivities and using the Hurewicz theorem, we recover the classical cubical homotopy excision theorem \citep{GoodwillieII} for simply connected spaces.
\begin{cor}
  If $A\colon \medsquare^3 \to \sSets$ is a strong homotopy pushout cube of simply connected spaces whose homotopy fibres $\hfib(A_\emptyset \to A_k)$ are $i_k$\hyph connected for $i_k \geq 1$, then the total fibre $\hfib(q\colon A_\emptyset \to \holim_{\righthalfcup^3}A)$ is $(i_1+i_2+i_3)$\hyph connected.
\end{cor}

\section{Setup and Notation}\label{sec:setup}
We will mainly follow the notations and conventions established in \citep{HEaC}. Our base category in this article is the category~$\sSets$ of (unpointed) simplicial sets, equipped with the standard Quillen (or Kan) model structure. Hence ``space'' will mean ``simplicial set'' throughout. As for cubical diagrams, given $n\in\N$, we write $\langle n\rangle \defas \{1,\ldots,n\}$,
\[ \medsquare^n \defas \Pfrak\langle n\rangle,\qquad \lefthalfcap^n \defas \medsquare^n\setminus\{\langle n\rangle\},\qquad \righthalfcup^n \defas \medsquare^n\setminus\{\emptyset\} \]
for the standard $n$\hyph cube, and the standard indexing posets for $n$\hyph dimensional pushouts and pullbacks respectively. Now a cubical diagram (of spaces) is just some diagram $X\colon P \to \sSets$ indexed by a poset $P \cong \medsquare^n$. Given such a $P \cong \medsquare^n$ as well as $p \leq q$ in~$P$, we define $\del_p^q \defas \set{r \in P}{p \leq r \leq q}$ and if~$X\colon P \to \sSets$ is a cubical diagram, we denote its faces by $\del_p^q X \defas X\vert_{\del_p^q}$. As particularly important special cases, given a single element $p \in P$, and denoting the bottom and top elements of~$P$ by~$\bot$ (e.g.~$\bot = \emptyset$ for $P = \medsquare^n$) and~$\top$ (e.g.~$\top = \langle n\rangle$ for $P = \medsquare^n$), respectively, we write $\del_p \defas \del_p^\top$ and $\del^p \defas \del_\bot^p$. Similarly for $\del_pX$ and $\del^p X$. For the standard $n$\hyph cube $\medsquare^n$ we allow ourselves the notational abuse of omitting braces in the face indices. So, for example, we will write~$\del_{1,2}$ instead of~$\del_{\{1,2\}}$.

Since every $P \cong \medsquare^n$ is a complete lattice, we use standard lattice notation, such as~$\bot$ for the bottom element, $\top$ for the top element and~$\neg$ for the complement. Again, for the standard $n$\hyph cube~$\medsquare^n$, we abusively write $\neg k$ or~$\hat k$ for $\langle n\rangle\setminus\{k\}$ instead of $\neg\{k\}$.

As is well-known, higher-dimensional homotopy pushouts (and dually pullbacks) can be calculated inductively as a series of ordinary two-dimensional homotopy pushouts. More explicitly, given a diagram $X\colon \lefthalfcap^{n+1} \to \sSets$ then
\[ \hocolim_{\lefthalfcap^{n+1}} X \simeq \hocolim\Bigl(\hocolim_{\lefthalfcap^{n}}\del_{n+1}X \from \hocolim_{\lefthalfcap^{n}}\del^{\neg n+1}X \to X_{\neg n+1}\Bigr), \]
which follows immediately from Thomason's theorem \citep[Theorem 26.8]{HToD} and the representation of~$\lefthalfcap^{n+1}$ as a Grothendieck construction
\[ \lefthalfcap^{n+1} \cong \int^{\lefthalfcap^2}\Bigl(\lefthalfcap^n \from \lefthalfcap^n \to \{\ast\}\Bigr). \]
Of course, since everything is symmetrical, it is not essential which pair of opposing faces is used for this calculation and we only chose $\del_{n+1}$ and $\del^{\neg n+1}$ as an example. In fact, it will play a crucial role later that we can pick any pair of opposing faces.

We can now reapply this to the two~$\lefthalfcap^n$ in the above Grothendieck construction etc. In the end, we can represent $\lefthalfcap^{n+1}$ by a Grothendieck construction involving only~$\lefthalfcap^2$ and points. One can use this to show that if all $2$\hyph faces of a cubical diagram are homotopy pushouts, then all higher-dimensional faces (as well as the entire cube) are homotopy pushouts. Such diagrams are called {\it strong homotopy pushouts} \citep{GoodwillieII}. Dually for {\it strong homotopy pullbacks}.

\section{Sets of Spaces}
Our basic objects of study are sets of spaces. Given two such sets of spaces $M$,~$N$, we write $N > M$ (``$N$ is killed by $M$''), which we call an acyclic inequality and means that every space in~$N$ becomes contractible upon left Bousfield localisation of~$\sSets$ at $\set{A \to \Delta[0]}{A\in M}$. If~$M$ or~$N$ are (equivalent to) singletons, we will usually omit the surrounding braces.
\begin{ex}
  As one can easily check, letting~$S^{-1}$ be the empty space, if $S^{-1} \in M$ then $N > M$ for all~$N$. Just as easily, one checks that $N > S^0$ iff all $X \in N$ are non-empty. More generally, one can show \citep{Farjoun1996} that $N > S^{n+1}$ iff every $X \in N$ is $n$\hyph connected.
\end{ex}

Equivalently \citep{Wojciech1996a,Wojciech1996b}, $N > M$ can be taken to mean that~$M$ is contained in the smallest class~$\bar\Ccal(N)$ containing~$N$ and closed under weak equivalences, (unpointed) homotopy colimits indexed by contractible categories and extensions by fibrations (whenever $F \to E \to B$ is a fibre sequence for any base point of~$B$ and $B$,~$F \in \bar\Ccal(N)$ then also~$E\in\bar\Ccal(N)$).

More restrictively, if we let $\Ccal(N)$ be the smallest class containing~$N$ and closed under just weak equivalences and homotopy colimits indexed by contractible categories, we define $N \gg M$ to mean that $M \subseteq \Ccal(N)$ \citep{Wojciech1996a,Wojciech1996b,Farjoun1996}. Obviously, $N \gg M$ implies $N > M$.

We frequently pass from single spaces to sets of spaces when taking homotopy fibres. Notationally, given $f\colon E \to B$, and $b \in B$, let us write $\hFib_b(f)$ for the homotopy fibre of~$f$ above~$b$ (or $\hFib_*(f)$ if~$b$ is clear from the context). On the other hand, we also write
\[ \Fib(f) \defas \set{\hFib_b(f)}{b \in B}. \]
As a final notation for the homotopy fibre, we simply write $\hFib(f)$ to indicate that all elements of~$\Fib(f)$ are weakly equivalent (e.g.~if~$B$ is connected) and that we just pick one representative.

\begin{rem}
  With this notation, the closure of a cellular class $\Ccal(A)$ under extensions by fibrations can be stated without mentioning base points. If~$\Ccal$ is a closed class (i.e.~closed under weak equivalences and pointed homotopy colimits), it is closed under extensions by fibrations iff for every map $f\colon E \to B$, having $B \in C$ and $\Fib(f) \subseteq\Ccal$ implies $E \in\Ccal$.
\end{rem}

Since we are working with unpointed spaces, taking loop spaces is not well-defined and our convention is that~$\Omega B \defas S^{-1}$ for~$B$ not connected. Otherwise, the homotopy type of the loop space is independent of the base point and we pick a representative. This means that~$\Omega B$ is only defined up to homotopy and not functorial. The usual loop space with respect to some $b\in B$ will by denoted by~$\Omega_b B$ or~$\Omega_* B$ if the base point is clear from the context.

Finally, two sets of spaces $M$,~$N$ are weakly equivalent iff every $X \in M$ is weakly equivalent to some $Y \in N$ and vice versa. Also, when applying homotopical constructions to sets of spaces, it is always understood that these are to be applied elementwise. For instance, if $M$,~$N$ are sets of spaces, their join is $M * N \defas \set{A * B}{A\in M,\, B\in N}$.

Lots of results, which are classically only formalised for single pointed or connected spaces generalise directly to arbitrary spaces when formalised using sets of spaces. Let us name a few of the most important ones here, which we are going to use. For more details see \citep{Werndli2016}.

Every composable pair of maps $f\colon A \to B$, $g\colon B \to C$ and every base point $b\in B$ gives rise to a fibre sequence $\hFib_b(f) \to \hFib_{g(b)}(g\circ f) \to \hFib_{g(b)}(g)$. When doing away with base points and taking sets of spaces, it doesn't make sense anymore to speak of fibration sequences but we still have an acyclic inequality
\begin{equation}\label{eqn:fibre sets of composable maps} \Fib(g\circ f) > \Fib(f) \cup \Fib(g). \end{equation}
As a special case of a composable pair, if a map $s\colon A \to B$ has a retraction~$r$, one obtains that $\hFib_b(s) \simeq \Omega_*\hFib_{r(b)}(r)$ for every $b \in B$. When taking fibre sets instead, we can no longer hope for a weak equivalence (just consider $s\colon \Delta[0] \to S^0$) but only cellular inequalities:
\begin{equation}\label{eqn:fibre sets and retracts} \Fib(r) \gg \Fib(s) \gg \Omega\Fib(r). \end{equation}

We will need the following result \citep[Appendix HL]{Farjoun1996} about the commutation of homotopy colimits and homotopy fibres on several occasions.
\begin{thm} {\bf (Puppe)}\label{thm:puppe}
  Given a diagram $X\colon \Ieu\to\sSets$ and $\tau\colon X \rar K$ a transformation to a constant diagram then
  \[ \hfib_k\bigl(\hocolim\tau\colon \hocolim X\to K\bigr) \simeq \hocolim(\hfib_k\tau) \]
  for every base point~$k$ of~$K$.
\end{thm}

\begin{cor}\label{cor:puppe}
  If $X\colon \medsquare^n \to \sSets$ is a strong homotopy pullback, the homotopy fibre of the comparison map $q\colon \hocolim X\vert_{\lefthalfcap^n} \to X_{\langle n\rangle}$ above $x\in X_{\langle n\rangle}$ is
  \[ \hfib_x(q) \simeq \bigast_{i=1}^n \hfib_x(X_{\hat\imath} \to X_{\langle n\rangle}),\qquad\text{whence}\qquad \Fib(q) \simeq \bigast_{i=1}^n \Fib\bigl(X_{\hat\imath} \to X_{\langle n \rangle}\bigr). \]
\end{cor}

\noindent On several occasions will we use the following generalisation of Chachólski's theorem \citep{Wojciech1996a,HEaC}.%
\begin{thm} {\bf (Chachólski)}\label{thm:chacholski}
  Every homotopy pushout square
  \[ \xymatrix@=1.5pc{ A \ar[r]^{f} \ar[d]_{g} & B \ar[d] \\ C \ar[r]_{g} & D } \]
  gives rise to a cellular inequality $\Fib(h) \gg \Fib(f)$.
\end{thm}

Finally, we also need the unsuspended square case of the acyclic homotopy excision theorem, as established in \citep{HEaC}.
\begin{thm}\label{thm:square}
  Given a homotopy pushout square as in Chachólski's theorem above with comparison map $q\colon A \to \holim(B \to D \from C)$, then
  \[ \Fib(q) > \Omega\Fib(f)\ast\Omega\Fib(g). \]
\end{thm}

\section{Thomason Magic}
In this section, we generalise a trick due to Chachólski \citep{Wojciech1997a}, who shows that, given a transformation of spans
  \[ \vcenter{\hbox{$\xymatrix@=1.5pc{ B \ar@{=}[d] & A \ar[l] \ar[r] \ar[d]_{f} & C \ar@{=}[d] \\ B & A' \ar[l] \ar[r] & C }$}}\qquad\vcenter{\hbox{with $\hocolim$}}\qquad\vcenter{\hbox{$\xymatrix@=1.5pc{P \ar[d]^{g} \\ Q & *!<2em,.5ex>{,}}$}} \]
one has $\Fib(\Sigma f) \gg \Fib(g) \gg \Sigma\Fib(f)$. The way he goes about the first cellular inequality is as follows:
\begin{enumerate}
  \item The suspension map $\Sigma f$ can be obtained as the homotopy colimit of
    \[ \xymatrix@=1.5pc{ \Delta[0] \ar@{=}[d] & B \ar[l]\ar[r]\ar@{=}[d] & B \ar@{=}[d] & A \ar[l]\ar[r]\ar[d]_{f} & C \ar@{=}[d] & C \ar[l]\ar[r]\ar@{=}[d] & \Delta[0] \ar@{=}[d] \\ \Delta[0] & B \ar[l]\ar[r] & B & A' \ar[l]\ar[r] & C & C \ar[l]\ar[r] & \Delta[0] & *!<2em,.5ex>{,} } \]
    (we can take the homotopy pushouts of the four outer appendages first, by Thomason's theorem \citep[Theorem 26.8]{HToD}). Instead, we can take the central homotopy pushouts first, and $\Fib(\Sigma f) \gg \Fib(g)$ then follows from Dror Farjoun's theorem \citep[Theorem 9.1]{Wojciech1996a} because all remaining fibres are contractible.
  \item Similarly, $g$ can be obtained as the homotopy colimit of
    \[ \xymatrix@=1.5pc{ B \ar@{=}[d] & A' \ar[l]\ar[r]\ar@{=}[d] & A' \ar@{=}[d] & A \ar[l]\ar[r]\ar[d]_{f} & A' \ar@{=}[d] & A' \ar[l]\ar[r]\ar@{=}[d] & C \ar@{=}[d] \\ B & A' \ar[l]\ar[r] & A' & A' \ar[l]\ar[r] & A' & A' \ar[l]\ar[r] & C & *!<2em,.5ex>{.} } \]
   Again, we can instead take the central pushout first. By Puppe's theorem, one has
    \[ \hFib_*\bigl(\hocolim(A' \from A \to A') \to A'\bigr) \simeq \hocolim\bigl(\Delta[0] \from \hFib_*(f) \to \Delta[0]\bigr) \]
    for every base point of~$A'$ and so $\hocolim(A' \from A \to A') \to A'$ is a {\it fibrewise suspension\/} of~$f$. All in all, the induced map between the central pushouts has $\Sigma\Fib(f)$ as its fibre set and $\Fib(g) \gg \Sigma\Fib(f)$ again follows from Dror Farjoun's theorem.
\end{enumerate}

This proof does not rely on our initial indexing poset being a span. All it does is to add appendages to all non-bottom elements, extend the diagram and apply Thomason's theorem in two different ways. For this reason, we refer to the result as {\it Thomason magic}. The adding of appendages is done by the following Grothendieck construction.
\begin{defn}
  Given a poset~$P$ with a bottom element~$\bot$, we define
  \[ \App(P) \defas \int^P F \qquad\text{where}\qquad F\colon P \to \Cat,\, p \mapsto \cases{ \{\emptyset\} & p = \bot \\ \lefthalfcap^2 & p \neq \bot } \]
  with $F(\bot \leq p)\colon \{\emptyset\} \to \lefthalfcap^2$ being the element~$\{1\}$ for~$p \neq \bot$ and all other morphisms being identities. This comes with an inclusion $P \hookrightarrow \App(P)$, $\bot \mapsto (\bot,\emptyset)$ and $p \mapsto (p,1)$ for $p \neq \bot$.  
\end{defn}

With this construction at hand, one can generalise the above proof in a straightforward manner even to arbitrary posets with a bottom element (see \citep{Werndli2016} for more details).
\begin{prop}\label{prop:thomason magic} {\bf (Thomason Magic)}
  Given a transformation $\tau\colon X \rar Y$ of diagrams $X$,~$Y\colon \lefthalfcap^n \to \sSets$ such that every component~$\tau_S$ with $S \neq \emptyset$ is a weak equivalence, then
  \begin{enumerate}
    \item $\Fib\bigl(\hocolim\tau\colon \hocolim X \to \hocolim Y\bigr) \gg \Sigma^{n-1}\Fib(\tau_\emptyset\colon X_\emptyset \to Y_\emptyset)$;
    \item $\Fib(\Sigma^{n-1}\tau_\emptyset\colon \Sigma^{n-1}X_\emptyset \to \Sigma^{n-1}Y_\emptyset) \gg \Fib\bigl(\hocolim\tau\colon \hocolim X \to \hocolim Y\bigr)$.
  \end{enumerate}
\end{prop}

\section{Serre's Theorem}
Another result we need to generalise to higher dimensions is the result \citep[Theorem 7.1]{Wojciech1997a}, which is referred to as the {\it generalised Serre theorem\/} in {\it op.~cit}.
\begin{lem}
  Let $\tau\colon X \rar Y$ be a transformation of diagrams $X$,~$Y\colon \lefthalfcap^n \to \sSets$ and $m \in \N_{\leq n}$ such that
  \[ \vcenter{\hbox{$\tau\vert_{\medsquare^{m-1}}\colon X\vert_{\medsquare^{m-1}} \rar Y\vert_{\medsquare^{m-1}}$}} \qquad \vcenter{\hbox{\begin{minipage}{15em} \begin{center} viewed as a diagram $\medsquare^m \to \sSets$ is a homotopy pushout \end{center}\end{minipage}}} \]
  and such that~$\tau$ is a weak equivalence outside of~$\medsquare^{m-1}$. Then $\tau_*\colon \hocolim_{\lefthalfcap^n} X \to \hocolim_{\lefthalfcap^n} Y$ is a weak equivalence.
\end{lem}
\begin{proof}
If~$n < 2$, the claim is trivial. The case $n = 2$ is in the proof of \citep[Theorem 7.1]{Wojciech1997a} and follows from a Fubini-type argument for homotopy colimits. The general case is by induction on~$n$. We use Thomason's theorem \citep[Theorem 26.8]{HToD} and get
\[ \xymatrix@C=1em@R=3ex{ \hocolim X \ar@{..>}[d] & \simeq & \hocolim\Bigl(\hocolim_{\lefthalfcap^{n-1}} X\vert_{\del_n} \ar@<2em>[d]_{\simeq} & \hocolim_{\lefthalfcap^{n-1}} X\vert_{\del^{\neg n}} \ar[l] \ar[r] \ar[d] & X_{\neg n} \ar@<-.5em>[d] \Bigr) \\ \hocolim Y & \simeq & \hocolim\Bigl(\hocolim_{\lefthalfcap^{n-1}} Y\vert_{\del_n} & \hocolim_{\lefthalfcap^{n-1}} Y\vert_{\del^{\neg n}} \ar[l] \ar[r] & Y_{\neg n} \Bigr) & *!<1em,1ex>{,} } \]
where the solid arrow on the left is a weak equivalences because $\del_n \subseteq \medsquare^n\setminus\medsquare^{m-1}$. If $m < n$, the same applies to the right arrow, while the middle one is a weak equivalence by the inductive hypothesis. Finally, if $m = n$, then, by the claim's hypotheses and Thomason's theorem, the right-hand square in the diagram is a pushout and the claim follows from the case $n=2$.
\end{proof}

\begin{thm} {\bf (Generalised Serre Theorem)}\label{thm:serre}
  Let $\tau\colon X \rar Z$ be a natural transformation of diagrams $X$,~$Z\colon \lefthalfcap^n \to \sSets$ and $m \in \N_{\leq n}$ such that
  \[ \vcenter{\hbox{$\tau\vert_{\medsquare^{m-1}}\colon X\vert_{\medsquare^{m-1}} \rar Z\vert_{\medsquare^{m-1}}$}} \qquad \vcenter{\hbox{\begin{minipage}{15em} \begin{center} viewed as a diagram $\medsquare^m \to \sSets$ is a strong homotopy pullback \end{center}\end{minipage}}} \]
  and such that~$\tau$ is a weak equivalence outside of~$\medsquare^{m-1}$. Furthermore, let us fix any base point in~$X_\emptyset$ and denote the fibres of the strong homotopy pullback~$\tau\vert_{\medsquare^{m-1}}$ by $F_1,\ldots,F_m$; i.e.
  \[ F_k \defas \hfib_*(X_\emptyset \to X_k) \text{ for $k < m$} \qquad\text{and}\qquad F_m \defas \hfib_*(X_\emptyset \to Z_\emptyset). \]
  Then, $\hfib_*(\hocolim X \to \hocolim Z) \gg \Sigma^{n-m}(F_1\ast\ldots\ast F_m)$.
\end{thm}
\begin{proof}
  The case $n = 2$ is \citep[Theorem 7.1]{Wojciech1997a}. For a general~$n$, we factor~$\tau$ as $X \rar Y \rar Z$, where~$Y$ agrees with~$Z$ everywhere except at~$Y_{\langle m-1\rangle}$, which is such that
  \[ \vcenter{\hbox{$\tau\vert_{\medsquare^{m-1}}\colon X\vert_{\medsquare^{m-1}} \rar Y\vert_{\medsquare^{m-1}}$}} \qquad \vcenter{\hbox{\begin{minipage}{15em} \begin{center} viewed as a diagram $\medsquare^m \to \sSets$ is a homotopy pushout. \end{center}\end{minipage}}} \]
  By the lemma above, the induced morphism $\hocolim X \to \hocolim Y$ is a weak equivalence and so it suffices to show that
  \[ \hfib_*(\hocolim Y \to \hocolim Z) \gg \Sigma^{n-m}(F_1\ast\ldots\ast F_m). \]
  $Y_{\langle m-1\rangle} \to Z_{\langle m-1\rangle}$ is the comparison map for the strong homotopy pullback $X\vert_{\medsquare^{m-1}} \rar Z\vert_{\medsquare^{m-1}}$, Puppe's theorem allows us to identify its fibre as being $F_1\ast\ldots\ast F_m$. So, what we are really going to show is that
  \[ \hfib_*(\hocolim Y \to \hocolim Z) \gg \Sigma^{n-m}\hfib_*\Bigl(Y_{\langle m-1\rangle} \to Z_{\langle m-1\rangle}\Bigr). \]
  We now repeatedly use Thomason's theorem to move this comparison map into a central location of some subcube and then apply Thomason magic. Identifying $\lefthalfcap^n \cong \int^{\lefthalfcap^2} (\lefthalfcap^{n-1} \from \lefthalfcap^{n-1} \to \{\ast\})$, we use Thomason's theorem to write the (induced map between) homotopy colimits as
\[ \xymatrix@C=1em@R=3ex{ \hocolim Y \ar@{..>}[d] & \simeq & \hocolim\Bigl(\hocolim_{\lefthalfcap^{n-1}} Y\vert_{\del_1} \ar@<2em>[d] & \hocolim_{\lefthalfcap^{n-1}} Y\vert_{\del^{\neg 1}} \ar[l] \ar[r] \ar[d]_{\simeq} & Y_{\neg 1} \ar@<-.5em>[d]_{\simeq} \Bigr) \\ \hocolim Z & \simeq & \hocolim\Bigl(\hocolim_{\lefthalfcap^{n-1}} Z\vert_{\del_1} & \hocolim_{\lefthalfcap^{n-1}} Z\vert_{\del^{\neg 1}} \ar[l] \ar[r] & Z_{\neg 1} \Bigr) & *!<1em,1ex>{.} } \]
  Observing that~$Y$ and~$Z$ agree everywhere except at~$\langle m-1\rangle$, the two solid arrows on the right are weak equivalences and using Dror-Farjoun's theorem \citep[Theorem 9.1]{Wojciech1996a}, it suffices to show
  \[ \hfib_*\Bigl(\hocolim_{\lefthalfcap^{n-1}} Y\vert_{\del_1} \to \hocolim_{\lefthalfcap^{n-1}}Z\vert_{\del_1}\Bigr) \gg \Sigma^{n-m}\hfib_*\Bigl(Y_{\langle m-1\rangle} \to Z_{\langle m-1\rangle}\Bigr). \]
  For this, we simply repeat the above, identify $\lefthalfcap^{n-1} \simeq \int^{\lefthalfcap^2} (\lefthalfcap^{n-2} \from \lefthalfcap^{n-2} \to \{\ast\})$ and use Thomason's theorem again etc. We do this $m-1$ times and noting that $\bigl(Y\vert_{\del_1}\bigr)\vert_{\del_1} = Y\vert_{\del_{1,2}} = Y\vert_{\del_{\langle 2\rangle}}$ (and similarly for~$\neg 1$, $Z$ and higher order restrictions), we end up needing to show that
  \[ \hfib_*\Bigl(\hocolim_{\lefthalfcap^{n-m+1}} Y\vert_{\del_{\langle m-1\rangle}} \to \hocolim_{\lefthalfcap^{n-m+1}}Z\vert_{\del_{\langle m-1\rangle}}\Bigr) \gg \Sigma^{n-m}\hfib_*\Bigl(Y_{\langle m-1\rangle} \to Z_{\langle m-1\rangle}\Bigr). \]
  But now $\langle m-1\rangle$ is the initial vertex of~$\del_{\langle m-1\rangle}$ and the claimed cellular inequality is exactly what we get from the Thomason magic \ref{prop:thomason magic}.
\end{proof}

\section{Suspended Comparison Map}\label{sec:suspended comparison map}
Starting with a strong homotopy pushout $A\colon \medsquare^n \to \sSets$, consider the comparison map $q\colon A_\emptyset \to P_\emptyset$, where $P\colon \medsquare^n \to \sSets$ agrees with~$A$ everywhere except at~$\emptyset$, where it is the homotopy pullback of~$A\vert_{\righthalfcup^n}$ (possibly replacing~$A$ fibrantly first). From Thomason magic~\ref{prop:thomason magic}, we know that
\[ \Fib(\Sigma^{n-1}q) \gg \Fib\left(A_{\langle n\rangle} \xto{s} \hocolim_{\lefthalfcap^n} P\right) \]
and so it suffices to understand this fibre set. Restricting the obvious natural transformations $A \rar P \rar \Const_{A_{\langle n\rangle}}$ to~$\lefthalfcap^n$ and taking homotopy colimits, we get that
\[ s\colon A_{\langle n\rangle} \to \hocolim_{\lefthalfcap^n} P \qquad\text{has a retraction}\qquad r\colon \hocolim_{\lefthalfcap^n} P \to A_{\langle n\rangle}. \]
But by \ref{eqn:fibre sets and retracts} $\Fib(s) \gg \Omega\Fib(r)$. Now, let $S\colon \medsquare^n \to \sSets$ be the strong homotopy pullback associated to~$A$ (i.e.~$S \defas \Rbb\Kan(A\vert_{\langle n\rangle^{\triangleright}})$, where $\langle n\rangle^{\triangleright} \subset \medsquare^n$ is the subposet of all $\hat k \to \langle n\rangle$. Composing the two obvious transformations $P \rar S$ and $S\rar\Const_{A_{\langle n\rangle}}$, restricting to~$\lefthalfcap^n$ and taking homotopy colimits, we get a factorisation
\[ \hocolim_{\lefthalfcap^n} P \xto{p} \hocolim_{\lefthalfcap^n} S \xto{p'} A_{\langle n\rangle} \qquad\text{of}\qquad r\colon \hocolim_{\lefthalfcap^n} P  \to A_{\langle n\rangle}. \]
The acyclic inequality \ref{eqn:fibre sets of composable maps} associated to this composable pair is
\begin{equation}\label{eq:suspended comparison map} \Fib(p'\circ p) > \Fib(p) \cup \Fib(p'). \end{equation}
Finally, Puppe's theorem \ref{thm:puppe} together with Chachólski's theorem \ref{thm:chacholski} yield
\[ \Fib(p') \simeq \bigast_{k=1}^n \Fib\bigl(A_{\hat k} \to A_{\langle n\rangle}\bigr) \gg \bigast_{k=1}^n \Ffrak(A_\emptyset \to A_k). \]
\begin{prop}\label{prop:suspended comparison map}
  Let $A\colon \medsquare^n \to \sSets$ be a strong homotopy pushout with corresponding homotopy pullback~$P$, strong homotopy pullback~$S$ and comparison maps $q\colon A_\emptyset \to P_\emptyset$, $p\colon \hocolim_{\lefthalfcap^n} P \to \hocolim_{\lefthalfcap^n} S$. Then we have a cellular inequality
  \[ \Fib(\Sigma^{n-1}q) > \Omega\left(\Fib(p)\cup\bigast_{k=1}^n \Fib(A_\emptyset \to A_k)\right). \]
  \vspace{-\baselineskip}\eop
\end{prop}

%

\section{The Main Theorem}
To better understand the map~$p$ from proposition \ref{prop:suspended comparison map}, we do not form the strong homotopy pullback~$S$ directly. Instead, we construct a sequence of homotopy pullback cubes
\begin{equation}\label{eqn:building strong homotopy pullbacks} P \simeq P^{(0)} \rar P^{(1)} \rar P^{(2)} \rar P^{(3)} \simeq S, \end{equation}
where each transformation induces a map between homotopy colimits over~$\lefthalfcap^3$ that we can understand. These three cubes correspond to the three pairs of opposite faces in~$\medsquare^3$. In this way, we subsequently replace all possible 2\hyph faces by homotopy pullbacks and therefore eventually arrive at a strong homotopy pullback.

\begin{defn}\label{defn:collected fibres}
  Given a cubical diagram~$X\colon \medsquare^n \to \sSets$ such that all~$X_M$ with $M \neq \emptyset$ are connected and $k \in \langle n\rangle$, we define
  \[ \Ffrak^X_k \defas \set{\hfib\Bigl(X_M \to X_{M\cup\{k\}}\Bigr)}{M \subseteq \langle n\rangle\setminus\{k\}} \]
  (where we omit the top index ``$X$'' if it is clear from the context). We call this the {\it collective $k^{\rm th}$ (homotopy) fibre\/}\index{Fibre!collective}\index{Collective!fibre}\index{Homotopy!fibre!collective} of~$X$.
\end{defn}

\begin{ex}
  If~$X$ is a strong pushout, from Chachólski's  theorem~\ref{thm:chacholski}, we know that $\Ffrak_k \gg F_k$ (and clearly $F_k \gg \Ffrak_k$). So, from a cellular viewpoint, $\Ffrak_k$ collapses to a single fibre.
\end{ex}

\begin{lem}
  Given a fibre sequence $F \to E \to B$ where $F > A$ and $E > \Sigma A$ for some simplicial set~$A$, then $B > \Sigma A$.
\end{lem}
\begin{proof}
  We form the cofibre sequence $E \to B \to B\sslash E$. Now $B\sslash E \gg \Sigma F$ by \citep[Proposition 10.5]{Wojciech1996b} and $\Fib(B \to B\sslash E) \gg E$ by \ref{thm:chacholski}, from which the claim follows.
\end{proof}

\begin{prop}\label{prop:cubical case 1}
  Let~$A\colon \medsquare^3 \to \sSets$ be a strong homotopy pushout of connected spaces and $F_k \defas \hfib(A_\emptyset \to A_k)$. If $P\colon \medsquare^n \to \sSets$ is the corresponding homotopy pullback, then $\Ffrak^P_k > \Sigma\Omega F_k$ for all $k \in \langle 3\rangle$ as long as $F_1$,~$F_2$ and~$F_3$ are connected.
\end{prop}
\begin{proof}
  Without loss of generality, let's assume $k = 1$. The collective fibres~$\Ffrak^A_1$ and~$\Ffrak^P_1$ are almost the same, with the former being killed by~$F_1$. The only fibre in~$\Ffrak^P_1$ that is not in~$\Ffrak^A_1$ is $\hfib(P_\emptyset \to A_1)$. Picking any base-point on~$A_\emptyset$ makes everything (and in particular~$P_\emptyset$) pointed. We now consider the fibre sequence
  \[ \hfib(A_\emptyset \to P_\emptyset)\ \longrightarrow\ \underbrace{\hfib(A_\emptyset \to A_1)}_{F_1}\ \longrightarrow\ \hfib(P_\emptyset \to A_1) \]
  associated to $A_\emptyset \to P_\emptyset \to A_1$.
  Because $F_1 \gg \Sigma\Omega F_1$ (see \citep[Corollary 10.6]{Wojciech1996b}) and using the lemma above, it suffices to show that $\hfib(A_\emptyset \to P_\emptyset) > \Omega F_1$. For this, we recall that~$P_\emptyset$ fits into a homotopy pullback square (of solid arrows)
  \[ \xymatrix@=1.5pc{ A_\emptyset \ar@{..>}[r] & P_\emptyset \ar[r]^-{f} \ar[d] & \holim(A_{1} \to A_{1,2} \from A_{2}) \ar[d]^{p} \\ & A_3 \ar[r]_-{g} & \holim(A_{1,3} \to A_{1,2,3} \from A_{2,3}). } \]
  The homotopy fibre of~$p$ above any base-point of $\holim(A_{1,3}\to A_{1,2,3}\from A_{2,3})$ is just the homotopy pullback
  \[ G \defas \holim\bigl(\hfib(A_1\to A_{1,3})\to\hfib(A_{1,2}\to A_{1,2,3})\from\hfib(A_2\to A_{2,3})\bigr). \]
  Since~$A$ is a strong homotopy pushout, all these fibres are killed by~$F_3$, which is connected, and hence connected themselves. The fibre sequence associated to the homotopy pullback~$G$ is
  \[ \Omega\hFib(A_{1,2}\to A_{1,2,3}) \to G \to \hFib(A_1\to A_{1,3})\times\hFib(A_2\to A_{2,3}) \]
  and since both the base and the fibre are killed by~$\Omega F_3$, it follows that also $G > \Omega F_3 > S_0$, so that~$\pi_0(G) \neq \emptyset$. Since the base-point was arbitrary, it follows that~$p$ hits all components of the base and hence $\Fib(g) \simeq \Fib(f)$. From acyclic homotopy excision for squares \ref{thm:square}, we know that
  \[ \Fib(g) > \Omega\hfib(A_3 \to A_{1,3})\ast\Omega\hfib(A_3 \to A_{2,3}) > \Omega F_1 \ast \Omega F_2 > \Sigma\Omega F_1, \]
  where we used that~$F_2$ is connected. By the same argument, we also have that
  \[ \Fib\bigl(A_\emptyset \to \holim(A_1 \to A_{1,2} \from A_2)\bigr) > \Omega F_1 \ast \Omega F_2 > \Sigma\Omega F_1. \]
  Now, the acyclic inequality \ref{eqn:fibre sets of composable maps} associated to the top composite in the last diagram gives
  \[ \hfib(A_{\emptyset} \to P_\emptyset) > \Fib\bigl(A_\emptyset \to \holim(A_1 \to A_{1,2} \from A_2)\bigr) \cup \Fib(f) \]
  and finally get $\hfib(A_\emptyset \to P_\emptyset) > \Omega\Sigma\Omega F_1 > \Omega F_1$.
\end{proof}

As already mentioned, we start with a homotopy pullback $P\colon \medsquare^3 \to \sSets$ and use the pairs of opposing faces $(\del_1,\del^{\neg 1})$,~$(\del_2,\del^{\neg 2})$ and~$(\del_3,\del^{\neg 3})$ of~$\medsquare^3$ to build the corresponding strong pullback in several steps, yielding a sequence \ref{eqn:building strong homotopy pullbacks} of homotopy pullback cubes. Pictorially, we are going to do the following (using $(\del_1,\del^{\neg 1})$ here):
\[  \xymatrix@C=2em@R=4ex{ & P_\emptyset \ar[rr] \ar[ld] \ar[dd]|\hole && P_1 \ar[dd] \ar[dl] \\
  P_2 \ar[rr] \ar[dd] && P_{1,2} \ar[dd] \\
  & P_3 \ar[rr]|(.475)\hole \ar[dl] && P_{1,3} \ar[dl] \\
  P_{2,3} \ar[rr] && P_{1,2,3}
  \ar@{}"1,4";"4,3"|(.2)*+{\displaystyle P^{(1)}_1}="Q1"
  \ar@{}"1,2";"4,1"|(.2)*+{\displaystyle P^{(1)}_\emptyset}="Q0"
  \ar"Q1";"2,3" \ar"Q1";"3,4" \ar@{..>}"1,4";"Q1"
  \ar"Q0";"2,1" \ar"Q0";"3,2"|(.3)\hole \ar@{..>}"1,2";"Q0" \ar@{..>}"Q0";"Q1"|(.79)\hole
}
\]
By definition, $P^{(1)}_\emptyset$ and~$P^{(1)}_1$ are the homotopy pullbacks of~$\del^{\neg 1}P$ and~$\del_1P$. Moreover, as described in section \ref{sec:setup}, by Thomason's theorem and using the hypothesis that~$P$ is a homotopy pullback cube, the square formed by $P_\emptyset$,~$P_1$,~$P^{(1)}_\emptyset$ and~$P^{(1)}_1$ is also a homotopy pullback. When passing to~$P^{(2)}$, we replace $P^{(1)}_\emptyset$ and~$P^{(1)}_2$ by the homotopy pullbacks of $\del^{\neg 2}P^{(1)}$ and~$\del_2P^{(1)}$, respectively. As the initial vertex changes, we need to make sure that the left face $\del^{\neg 1}$ stays a homotopy pullback. Since everything is symmetric, we might just as well check that in the cubical diagram above, if the back and front faces are homotopy pullbacks, so is the square containing $P^{(1)}_\emptyset$,~$P^{(1)}_1$,~$P_3$ and~$P_{1,3}$. This follows from Thomason's theorem, by which
\[ \xymatrix@=1.5pc{ P^{(1)}_\emptyset \ar[r] \ar[d] & \holim\bigl(P_3 \to P_{1,3} \from P_1^{(1)}\bigr) \ar[d] \\ P_2 \ar[r] & \holim\bigl(P_{2,3} \to P_{1,2,3}  \from P_{1,2}\bigr) } \]
is a homotopy pullback and the bottom map is a weak equivalence by assumption. Similarly when passing from~$P^{(2)}$ to~$P^{(3)}$, where then, the cube has become the strong homotopy pullback obtained from all the $P_{\hat k} \to P_{\langle3\rangle}$ because all 2\hyph faces of~$P^{(3)}$ are homotopy pullbacks.

\begin{prop}\label{prop:cubical case 2}
  Let~$P\colon \medsquare^3 \to \sSets$ be a homotopy pullback cube and~$P'$ the new homotopy pullback cube obtained from~$P$ by replacing~$P_\emptyset$ and~$P_k$ with the homotopy pullbacks of~$\del^{\neg k}$ and~$\del_k$ for some $k \in \langle 3\rangle$. If all~$P_M$ and~$P'_M$ with $M \neq \emptyset$ are connected, then $\Ffrak_l^{P'} \subseteq \Ffrak_l^{P}$ (up to weak equivalences) for all $l \in \langle 3\rangle$, whence $\Ffrak_l^{P'} \gg \Ffrak_l^{P}$.%
\end{prop}
\begin{proof}
  Without loss of generality, let's assume $k=1$ (i.e.~$P' = P^{(1)}$). The only vertices, where~$P$ and~$P^{(1)}$ differ are those at~$\emptyset$ and~$1$ and so, the only fibres that change are those of the morphism associated to $\emptyset \to 1$ (in~$\Ffrak_1$) as well as those associated to $\emptyset \to \{l\}$ and $1 \to \{1,l\}$ (in~$\Ffrak_l$) for $1 \neq l$. The first one is easy because
  \[ \hfib\left(P^{(1)}_\emptyset \to P^{(1)}_1\right) \simeq \hfib\bigl(P_\emptyset \to P_1\bigr) \]
  (because the corresponding square is a homotopy pullback) and therefore $\Ffrak^{P'}_1 \simeq \Ffrak^{P}_1$. As for the other ones, we just use that~$P^{(1)}_\emptyset$ and~$P^{(1)}_1$ are obtained as the pullbacks of the left and right faces in the cube above. So, for example, taking $l = 2$,
  \[ \hfib\left(P^{(1)}_\emptyset \to P_2\right) \simeq \hfib(P_3 \to P_{2,3}), \quad \hfib\left(P^{(1)}_1\to P_{1,2}\right) \simeq \hfib(P_{1,3} \to P_{1,2,3}). \]
  So $\Ffrak^{P'}_2$ comprises just these two fibres and is hence contained in~$\Ffrak^P_2$.
\end{proof}

\begin{prop}\label{prop:cubical case 3}
  Let $P$,~$P'\colon \medsquare^3 \to \sSets$ and $k \in \langle 3\rangle$ as in the last proposition. If all~$P_M$ and~$P'_M$ with $M \neq \emptyset$ are connected, then the canonical transformation $P \rar P'$ induces a map
  \[ \hocolim_{\lefthalfcap^3} P \to \hocolim_{\lefthalfcap^3} P', \]
  whose homotopy fibre~$F$ satisfies $F \gg \Sigma\bigl(\Ffrak^P_k\ast\hfib(P_k \to P'_k)\bigr)$.
\end{prop}
\begin{rem}
  Implicit in our proposition is that~$\hocolim_{\lefthalfcap^3}P'$ is connected as well, which follows readily from the Mayer-Vietoris long exact sequence for~$\hocolim_{\lefthalfcap^3} P'$.
\end{rem}
\begin{proof}
  Without loss of generality, let's treat the case $k = 1$ (i.e.~$P' = P^{(1)}$). Note that $P \rar P'$ consists of identities everywhere except at~$\emptyset$ and~$1$, where we have the homotopy pullback square
  \[ \xymatrix@=1.5pc{ P_\emptyset \ar[r] \ar[d] & P_1 \ar[d] \\ P'_\emptyset \ar[r] & P'_1 & *!<2em,.5ex>{.} } \]
  By the generalised Serre theorem \ref{thm:serre}, we get that the homotopy fibre~$F$ of the induced map between homotopy colimits satisfies
  \[ F \gg \Sigma\Bigl(\underbrace{\hfib(P_\emptyset \to P_1)}_{\in\Ffrak^P_1}\ast\underbrace{\hfib\bigl(P_\emptyset \to P'_\emptyset\bigr)}_{\simeq\hfib(P_1 \to P'_1)}\Bigr) \]
\end{proof}

\begin{cor}
  Let $P\colon \medsquare^3 \to \sSets$ be a homotopy pullback with corresponding strong homotopy pullback $S\colon \medsquare^3 \to \sSets$ and comparison map $p\colon \hocolim_{\lefthalfcap^3} P \to \hocolim_{\lefthalfcap^3} S$. If all~$P_M$ and~$S_M$ with~$M \neq \emptyset$ are connected, then
  \[ \hfib(p) \gg \Sigma\set{\Ffrak^P_k\ast\hfib(P_k \to S_k)}{k\in\langle3\rangle}. \]
\end{cor}
\begin{proof}
  We construct~$S$ in several steps, as outlined in \ref{eqn:building strong homotopy pullbacks} and note that $P^{(1)}_1 \simeq S_1$, $P^{(2)}_2 \simeq S_2$ and $P^{(3)}_3 \simeq S_3$, while $P^{(1)}_2 \simeq P_2$ and $P^{(2)}_3 \simeq P^{(1)}_3 \simeq P_3$. Now, writing
  \[ p_k\colon \hocolim_{\lefthalfcap^3} P^{(k-1)} \to \hocolim_{\lefthalfcap^3} P^{(k)} \qquad\text{(where $P^{(0)} \defas P$)} \]
  for the canonical maps, we have $p = p_3\circ p_2\circ p_1$. By the acyclic inequality \ref{eqn:fibre sets of composable maps} for composable pairs and \ref{prop:cubical case 3}, we have

  \[ \hfib(p) > \set{\hfib(p_k)}{k\in\langle3\rangle} \gg \set{\Sigma\left(\Ffrak^{P^{(k-1)}}_k\ast\hfib(P_k \to S_k)\right)}{k\in\langle 3\rangle} \]
  Using~\ref{prop:cubical case 2}, we can now replace~$\Ffrak^{P^{(1)}}_2$ and~$\Ffrak^{P^{(2)}}_3$ by~$\Ffrak^P_2$ and~$\Ffrak^P_3$, respectively.
\end{proof}

Finally, by combining this corollary with \ref{prop:cubical case 1} and the square case \ref{thm:square}, we obtain the following weak version of the 3\hyph dimensional homotopy excision theorem.
\begin{thm}
  Let $A\colon \medsquare^3 \to \sSets$ be a strong homotopy pushout of connected spaces, with homotopy fibres $F_k \defas \hfib(A_\emptyset \to A_k)$ and comparison map $q\colon A_\emptyset \to \holim_{\righthalfcup^3} A$. As long as the homotopy fibres $F_1$,~$F_2$ and~$F_3$ are again connected,
  \[ \hFib(\Sigma^2q) > \Sigma(\Omega F_1\ast\Omega F_2\ast\Omega F_3). \]
\end{thm}
\begin{proof}
  Let $P\colon \medsquare^3 \to \sSets$ be the cubical homotopy pullback associated to~$A$ and $S\colon \medsquare^3 \to \sSets$ the strong homotopy pullback of all the $A_{\hat{k}} \to A_{\langle 3\rangle}$. We first note that all~$P_M$ and~$S_M$ with $M \neq \emptyset$ are connected. Indeed, $P_1 \to P^{(1)}_1$ has no empty homotopy fibres because
  \[ \Fib\left(P_1 \to P^{(1)}_1\right) \stackrel{\text{\tiny \ref{thm:square}}}{>} \Omega\hfib(P_1 \to P_{1,2}) * \Omega\hfib(P_1 \to P_{1,3}) \stackrel{\text{\tiny \ref{thm:chacholski}}}{>} \Omega F_2 * \Omega F_3 > S^1. \]
  Hence all components are hit and since $P_1 = A_1$ is connected, so is~$P^{(1)}_1$. Similarly when passing to~$P^{(2)}$ and~$P^{(3)} = S$. Writing $p\colon \hocolim_{\lefthalfcap^3} P \to \hocolim_{\lefthalfcap^3} S$, it suffices to show that $\hfib(p) > \Sigma^2(\Omega F_1\ast\Omega F_2\ast\Omega F_3)$ because then, by~\ref{prop:suspended comparison map},
  \[ \hfib(\Sigma^2q) \gg \Omega\hfib(p) > \Omega\Sigma^2(\Omega F_1\ast\Omega F_2\ast\Omega F_3) > \Sigma(\Omega F_1\ast\Omega F_2\ast\Omega F_3). \]
  From the above corollary (and noting that $P_k = A_k$), we already know that 
  \[ \hfib(p) \gg \Sigma\set{\Ffrak^P_k\ast\hfib(A_k \to S_k)}{k\in\langle3\rangle} \]
  and by~\ref{prop:cubical case 1}, we then have
  \[ \hfib(p) > \Sigma\set{\Sigma\Omega F_k\ast\hfib(A_k \to S_k)}{k\in\langle3\rangle}. \]
  But by definition $S_1 = \holim(A_{1,2} \to A_{1,2,3} \from A_{1,3})$, so that, by the acyclic homotopy excision theorem for squares \ref{thm:square} and Chachólski's theorem \ref{thm:chacholski},
  \[ \hfib(A_1 \to S_1) > \Omega\hfib(A_1\to A_{1,2}) \ast \Omega\hfib(A_1\to A_{1,3}) > \Omega F_2 \ast \Omega F_3. \]
  Similarly for the other two homotopy fibres, so that, all in all,
  \[ \hfib(p) > \Sigma^2(\Omega F_1\ast\Omega F_2\ast\Omega F_3). \]
\end{proof}

Combining this theorem with the relative Hurewicz isomorphism theorem \citep[Theorem 7.5.4]{Spanier1994}, we recover a cubical version of the classical homotopy excision theorem \citep{GoodwillieII} for simply connected spaces.
\begin{cor}
  If $A\colon \medsquare^3 \to \sSets$ is a strong homotopy pushout cube of simply connected spaces whose homotopy fibres $\hfib(A_\emptyset \to A_k)$ are $i_k$\hyph connected for $i_k \geq 1$, then the total fibre $\hfib(q\colon A_\emptyset \to \holim_{\righthalfcup^3}A)$ is $(i_1+i_2+i_3)$\hyph connected.
\end{cor}
\begin{proof}
  This follows from the relative Hurewicz isomorphism theorem together with the cellular inequality \citep[Proposition 10.5]{Wojciech1996b}
  \[ \Cof(\Sigma^2q) > \Sigma\hFib(\Sigma^2q). \]
\end{proof}

\bibliographystyle{mybst.bst}
\bibliography{homotopy.bib}

\begin{thebibliography}{15}
\providecommand{\natexlab}[1]{#1}
\providecommand{\url}[1]{\texttt{#1}}
\providecommand{\urlprefix}{URL }

\bibitem[{Blakers and Massey(1952)}]{Blakers1952}
Blakers, A.~L.; Massey, W.~S.; \textit{The homotopy groups of a triad. {II}};
  Ann. of Math. (2); vol.~55 pp. 192--201; 1952

\bibitem[{Bousfield(1994)}]{Bousfield1994}
Bousfield, A.~K.; \textit{Localization and periodicity in unstable homotopy
  theory}; J. Amer. Math. Soc.; vol.~7, no.~4 pp. 831--873; 1994

\bibitem[{Brown and Loday(1987)}]{Brown1987}
Brown, R.; Loday, J.-L.; \textit{Homotopical excision, and {H}urewicz theorems
  for {$n$}-cubes of spaces}; Proc. London Math. Soc. (3); vol.~54, no.~1 pp.
  176--192; 1987

\bibitem[{Chach{\'o}lski(1996{\natexlab{a}})}]{Wojciech1996a}
Chach{\'o}lski, W.; \textit{Closed classes}; in \textit{Algebraic topology: new
  trends in localization and periodicity ({S}ant {F}eliu de {G}u\'\i xols,
  1994)}; \textit{Progr. Math.}, vol. 136; pp. 95--118; Birkh\"auser, Basel;
  1996{\natexlab{a}}

\bibitem[{Chach{\'o}lski(1996{\natexlab{b}})}]{Wojciech1996b}
Chach{\'o}lski, W.; \textit{On the functors\/ {$CW_A$} and\/ {$P_A$}}; Duke
  Math. J.; vol.~84, no.~3 pp. 599--631; 1996{\natexlab{b}}

\bibitem[{Chach{\'o}lski(1997)}]{Wojciech1997a}
Chach{\'o}lski, W.; \textit{A generalization of the triad theorem of
  {B}lakers-{M}assey}; Topology; vol.~36, no.~6 pp. 1381--1400; 1997

\bibitem[{Chach{\'o}lski et~al.(2015)Chach{\'o}lski, Farjoun, Flores, and
  Scherer}]{Chacholski2015}
Chach{\'o}lski, W.; Farjoun, E.~D.; Flores, R.; Scherer, J.; \textit{Cellular
  properties of nilpotent spaces}; Geom. Topol.; , no.~5 pp. 2741--2766; 2015

\bibitem[{Chach{\'o}lski and Scherer(2002)}]{HToD}
Chach{\'o}lski, W.; Scherer, J.; \textit{Homotopy theory of diagrams}; Mem.
  Amer. Math. Soc.; vol. 155, no. 736 pp. x+90; 2002

\bibitem[{Chach{\'o}lski et~al.(2014)Chach{\'o}lski, Scherer, and
  Werndli}]{HEaC}
Chach{\'o}lski, W.; Scherer, J.; Werndli, K.; \textit{Homotopy Excision and
  Cellularity}; 2014; \urlprefix\url{http://arxiv.org/abs/1408.3252}; to appear
  in {\it Annales de l'Institut Fourier}

\bibitem[{Ellis and Steiner(1987)}]{Ellis1987}
Ellis, G.; Steiner, R.; \textit{Higher-dimensional crossed modules and the
  homotopy groups of {$(n+1)$}-ads}; J. Pure Appl. Algebra; vol.~46, no. 2-3
  pp. 117--136; 1987

\bibitem[{Farjoun(1996)}]{Farjoun1996}
Farjoun, E.~D.; \textit{Cellular spaces, null spaces and homotopy
  localization}; \textit{Lecture Notes in Mathematics}, vol. 1622;
  Springer-Verlag, Berlin; 1996

\bibitem[{Goodwillie(1992)}]{GoodwillieII}
Goodwillie, T.~G.; \textit{Calculus. {II}. {A}nalytic functors}; $K$-Theory;
  vol.~5, no.~4 pp. 295--332; 1992

\bibitem[{Munson and Volić(2015)}]{Munson2015}
Munson, B.; Volić, K.; \textit{Cubical Homotopy Theory}; \textit{New
  Mathematical Monographs}, vol.~28; Cambridge Univ. Press, Cambridge; 2015

\bibitem[{Spanier(1994)}]{Spanier1994}
Spanier, E.~H.; \textit{Algebraic topology}; Springer-Verlag, New York-Berlin;
  1994

\bibitem[{Werndli(2016)}]{Werndli2016}
Werndli, K.; \textit{Cellular Homotopy Excision}; Ph.D. thesis; EPFL; 2016

\end{thebibliography}

\end{document}